\documentclass[a4paper]{amsart}
\usepackage[T1]{fontenc}
\usepackage{amssymb}
\usepackage[all]{xy}
\usepackage{verbatim}

\newcommand{\fun}[3]{#1\colon #2\rightarrow #3}

\newcommand{\set}[2]{\{#1:\, #2\}}

\newcommand{\psr}[2]{#1[[#2]]}

\newcommand{\abs}[1]{\lvert#1\rvert}
\newcommand{\absb}[1]{\bigl\lvert#1\bigr\rvert}

\newcommand{\bb}[1]{\mathbb{#1}}
\newcommand{\bff}[1]{\mathbf{#1}}

\newcommand{\dotsk}[0]{,\dots,}

\newcommand{\field}[0]{k}
\newcommand{\f}[0]{k}                     

\newcommand{\bbN}[0]{\mathbb{N}}
\newcommand{\bbFp}[0]{\mathbb{F}_p}
\newcommand{\bbFq}[0]{\mathbb{F}_q}
\newcommand{\bbZp}[0]{\mathbb{Z}_p}

\newcommand{\bbQp}[0]{\mathbb{Q}_p}

\newcommand{\bbF}[0]{\mathbb{F}}
\newcommand{\bbL}[0]{\mathbb{L}}
\newcommand{\bfW}[0]{\mathbf{\W}}
\newcommand{\bbZ}[0]{\mathbb{Z}}
\newcommand{\bbA}[0]{\mathbb{A}}
\newcommand{\bbQ}[0]{\mathbb{Q}}

\newcommand{\DVR}[0]{\mathcal{O}}          
\newcommand{\SSDVR}[0]{\mathcal{X}}          

\newcommand{\W}[0]{\mathbf{W}}          

\newcommand{\catring}[0]{\mathsf{Rings}}                
\newcommand{\catvar}[0]{\mathsf{Var}}                

\DeclareMathOperator{\KO}{K_0}
\DeclareMathOperator{\ord}{ord}

\DeclareMathOperator{\spec}{Spec}

\DeclareMathOperator{\R}{r}                             
\DeclareMathOperator{\V}{V}                             
\DeclareMathOperator{\F}{F}                             

\newcommand{\muH}{\mu_{Haar}}

\DeclareMathOperator{\symdiff}{\Delta}

\newcommand{\KOR}{\KO(\catvar_{\f})}
\newcommand{\KOLunmarked}{\mathcal{M}}     
\newcommand{\KOL}{\mathcal{M}_{\f}}                                  
\newcommand{\KOLp}{\mathcal{M}_{\bbFp}}  
\newcommand{\KOW}{\overline{\KO}(\catvar_{\f})}
\newcommand{\KOWp}{\overline{\KO}(\catvar_{\bbFp})}
\newcommand{\KOWq}{\overline{\KO}(\catvar_{\bbFq})}
\newcommand{\KOC}{\widehat{\KOL}}

\DeclareMathOperator{\counting}{C}

\DeclareMathOperator{\filtration}{F}

\numberwithin{equation}{section}

\newtheorem{theorem}[equation]{Theorem}
\newtheorem{lemma}[equation]{Lemma}
\newtheorem{corollary}[equation]{Corollary}
\newtheorem{proposition}[equation]{Proposition}
\newtheorem{example}[equation]{Example}

\newtheorem{remark}[equation]{Remark}
\newtheorem{definition-lemma}[equation]{Definition-Lemma}

\theoremstyle{remark}
\newtheorem*{acknowledgements}{Acknowledgment}

\begin{document}
\title[Computing $p$-adic integrals using motivic integration]{Computing $p$-adic integrals using motivic integration}
\author{Karl R\"okaeus}
\address{Karl R\"okaeus \\ Department of Mathematics \\ Stockholm University \\ SE-106 91 Stockholm \\ Sweden}
\email{karlr@math.su.se}
\date{October 6, 2008}
\begin{abstract}
We use the theory of motivic integration in order to give a geometric explanation of the behavior of some $p$-adic integrals.
\end{abstract}
\maketitle

\section{Introduction}
In \cite{Skoglund2005} the author computes the integral $I_p=\int_{\bbZp^n}\abs{\prod_{1\leq i<j\leq n}(X_i-X_j)}_pd\mu_{Haar}$. The method is recursive and finds a rational function $f$ such that $I_p=f(p)$ (with respect to the normalized absolute value).
Looking at these computations it is immediate that instead of integrating over $\bbZp^n$ we may as well integrate over $\W(\bbFq)^n$, where $\W$ is the Witt vectors and $q$ is any power of $p$, to obtain the value of $I_q=\int_{\W(\bbFq)^n}\abs{\prod_{1\leq i<j\leq n}(X_i-X_j)}_pd\mu_{Haar}$ for different powers  $q$ of the fixed prime $p$. It turns out that for every $q$, $I_q=f(q)$. Here $f$ is the same rational function as above, it is hence independent of $q$.

This kind of behavior is unusual, one would expect that $f$ should vary with $q$ (\emph{cf.} Example \ref{MC9}). When it occurs one could suspect that there is some geometric explanation; the aim of this paper is to give such an explanation.


For this we use a version of geometric motivic integration, developed in \cite{MF}. This theory is constructed in almost exactly the same way as the standard theory of geometric motivic integration (see \cite{MR1886763} or \cite{MR1905024}, and \cite{MR2075915} for the case of equal and mixed characteristic respectively). The main difference is that when the standard motivic measure takes values in $\KOC$, i.e., $\KOL=\KOR[\bbL^{-1}]$ completed with respect to the dimension filtration, we complete $\KOL$ with respect to a stronger topology. This allows us to specialize our motivic integral to the corresponding $p$-adic one, by counting $\bbFp$-points on the motivic integral. On the contrary, the point counting homomorphism is not continuous with respect to the dimension filtration, hence does not extend to $\KOC$.

For this reason we defined, in \cite{MF}, a topology on $\KOL$ for which the point counting homomorphism is continuous in case $\f=\bbFp$. Let $\KOW$ be $\KOL$ completed with respect to this topology. Let, for every power $q$ of $p$, $\counting_q$ be the point counting homomorphism $\KOLp\to\bbQ$, induced by $X\mapsto\abs{X(\bbFq)}$. It extends by continuity to a homomorphism $\counting_q\colon\KOWp\to\bb{R}$. Now, when $\SSDVR$ is a variety over a complete discrete valuation ring, we may define a measure $\mu_{\SSDVR}$ on certain subsets of its arc space, $\SSDVR_\infty$. This measure takes values in $\KOW$, where $\f$ is the residue field of the discrete valuation ring. (Instead of taking values in $\KOC$, which is the standard choice.)  In case $\SSDVR=\bbA_{\bbZp}^n$ we will have $\SSDVR_\infty=\W(\overline{\bb{F}}_p)^n$, and $\counting_q\mu_{\SSDVR}(A)=\mu_{Haar}(A\cap\W(\bbFq)^n)$ for every measurable subset $A\subset\SSDVR_\infty$. 

Using this theory of motivic integration, we prove that if $\SSDVR=\bbA_{\bbZp}^n$ then 
$$I=\int_{\SSDVR_\infty}\abs{\prod_{1\leq i<j\leq n}(X_i-X_j)}_pd\mu_{\SSDVR}\in\KOWp$$
is a rational function in $\bbL=[\bbA_{\bbFp}]$ with integer coefficients, more precisely $I=f(\bbL)$ where $f$ is the same as above (this follows from Theorem \ref{mainMC}). By applying the point counting homomorphism we recover the original integrals, $\counting_qI=I_q$ for every power $q$ of $p$, showing that $I_q=f(q)$.


Here is an overview of the paper: In Section \ref{25} we give a quick review of the theory from \cite{MF} in the case of interest to us in this paper, namely when $\SSDVR$ is an affine space. We then prove some general result about our motivic integral in this case, in particular that all the integrals that we are interested in exist. 

In Section \ref{24} we compute the integral of the absolute value of a polynomial in one variable. This case is easier than the case of many variables, and it is possibly to compute a satisfactory explicit value of the integral for any polynomial whose image over the residue field is separable.

In Section \ref{MC8} we give three fundamental results, about change of variables and separation of variables. We need these in the last section, to perform the recursions which are the ultimate goal of the paper.

In Section \ref{MC3} we do the work that allow us to compute the integral $I$ defined above. In fact, we compute more general integrals; we show that for $p$ sufficiently large, the motivic integral of the absolute value of any product of linear forms is a rational function in $\bbL$, with coefficients in $\bbZ$. Then when $p$ is arbitrary we compute the motivic integral of the absolute value of a product of more special linear forms, in particular the integral $I$. These computations also make it possible to give an explicit formula for the integral. (If one is only interested in the $p$-adic integrals, it is possible to translate these computations to that setting, with no reference to motivic integration.) We will also see that all these computations work also in the equal characteristic case.

This paper includes a fair amount of computation in the Witt vectors, we therefore include an appendix on them.
\begin{acknowledgements}
The author is grateful to Professor Torsten Ekedahl for various suggestions and  comments on the paper.
\end{acknowledgements}

\section{Background and general results}\label{25}

In \cite{MF}, we developed a version of motivic integration, valid over any variety $\SSDVR$ defined over a complete discrete valuation ring, $\DVR$. In this paper we will use this theory in the case when $\SSDVR$ is an affine space, $\SSDVR=\bbA_{\DVR}^d$. Actually, what interests us the most is when $\DVR=\bbZp$, for then the computations of certain $p$-adic integrals can be done in the motivic setting, and then be obtained by applying $\counting_p$ to the result.

The present section is divided into two parts: In the first, we give a brief review, without proofs, of the theory from \cite{MF} in the special case when $\SSDVR=\bbA_{\DVR}^d$. Then in the second part we prove that all the integrals we are interested actually exist, and prove some basic facts about them.

\subsection{Basic definitions}\label{MC-103}

We write $\KOL$ to denote $\KOR[\bbL^{-1}]$, the Grothendieck ring of $\f$-varieties with the class of the affine line inverted. Our motivic integral takes values in a certain completion of $\KOL$, which we denote $\KOW$ (here, $\f$ is the residue field of $\DVR$). For the exact definition of the topology which we are using we refer to \cite{MF}, here we briefly describe some of its basic properties: It is stronger than the topology coming from the dimension filtration. In fact $\KOW$ can be embedded in $\KOC$, however, its topology is not the subspace topology. We do not have the ultra metric properties of $\KOC$. Instead the point counting homomorphisms $\counting_q\colon\KOWq\to\bb{R}$, induced by $X\mapsto\abs{X(\bbFq)}$, are well-defined and continuous. In this paper the questions of convergence will mainly boil down to the following result, which is Example 2.14 in \cite{MF}:

\begin{proposition}\label{MC-34}
Let $\f$ be any field. If $\{e_i\}_{i\in\bbN}$ is a sequence of integers such that $e_i\to\infty$, then $\sum_{i\in\bbN}\bbL^{-e_i}$ is convergent in $\KOW$. 
\end{proposition}
It follows for example that $1-\bbL^{-1}$ is invertible. On the other hand, $\sum_{n\geq 0}3^n\bbL^{-n}$ is not convergent; $1-3\bbL^{-1}$ is not invertible (contrary to with respect to the dimension filtration).

We have a partial ordering on $\KOW$, defined on $\KOR$ as $x\leq y$ if $x+[X]=y$ for some scheme $X$, which we may often use to reduce questions about convergence to Proposition \ref{MC-34}.

The construction of the motivic measure theory in \cite{MF} is the standard one, used in the appendix of \cite{MR1905024}, in \cite{MR1886763} and in \cite{MR2075915}. We will not go through these definitions here, we just tell what the results are in the case of interest in this paper.

From now on, we assume that $\SSDVR=\bbA_{\DVR}^d$ where $\DVR$ is a complete discrete valuation ring with residue field $\f$, of one of the following types:
\begin{itemize}
\item $\f$ is a perfect field, of prime characteristic $p$, and $\DVR=\W(\f)$, where $\W$ is the ring scheme of Witt vectors constructed with respect to the prime $p$. 
\item $\f$ is a field, and $\DVR=\psr{\f}{t}$.
\end{itemize}
(We are mainly interested in the cases when $\DVR$ is either $\bbZp$ or $\psr{\bbQ}{t}$ so that $\f=\bbFp$ and $\bbQ$ respectively.)  Then the space of arcs $\SSDVR_\infty$ of $\SSDVR=\bbA^d_{\DVR}$ is the set
$$\SSDVR_\infty=\begin{cases}\W(\overline{\f})^d, &  \DVR=\W(\f)\\ \psr{\overline{\f}}{t}^d, & \DVR=\psr{\f}{t} \end{cases},$$
where $\overline{\f}$ is an algebraic closure of $\f$. We use the standard terminology in connection with the Witt vectors, see the appendix. Moreover, in order to get a uniform notation, we will write $(x_i)_{i\geq0}$ for $\sum_{i\geq0}x_it^i\in \psr{\overline{\f}}{t}$.

For every $n$, if $\DVR=\W(\f)$ then $\SSDVR_n$ is the scheme whose $R$-points is $\W_n^d(R)$ for every  $\f$-algebra $R$, where $\W_n$ is the Witt vectors of length $n$. If instead $\DVR=\psr{\f}{t}$ then $\SSDVR_n(R)=(R[t]/(t^n))^d$ for every $\f$-algebra $R$. We use $\pi_n$ to denote the projection $\SSDVR_\infty\to\SSDVR_n(\overline{\f})$.

The definitions of stable and measurable subsets of $\SSDVR$ are the standard ones. Moreover, the measure of a subset $A\subset\SSDVR_\infty$ which is stable of level $n$ is $\mu_{\SSDVR}(A):=[\pi_n(A)]\bbL^{-nd}\in\KOW$. The measure of a general measureable set, and of a measurable function, is defined in the standard way. The following (Propositions $3.9$ and $3.11$ in \cite{MF}) are the properties that we are after:
\begin{proposition}\label{MC39}
Let $\SSDVR=\bbA_{\bbZp}^d$. When $A\subset\SSDVR_\infty$ is measurable we have, for every power $q$ of $p$, $\mu_{Haar}(A\cap\W(\bbFq)^d)=\counting_q(\mu_{\SSDVR}(A))$. Moreover, if $f\colon\SSDVR_\infty\to\KOW$ is measurable, then $\counting_q\int_{\SSDVR_\infty}fd\mu_{\SSDVR}=\int_{\W(\bbFq)^d}\counting_q\circ fd\mu_{Haar}$.
\end{proposition}


\subsection{Motivic integrals of absolute values of polynomials}
We continue to use $\DVR$ to denote a complete discrete valuation ring with residue field $\f$.

To show that the integrals we are interested in exists we need the following simple approximation:
\begin{lemma}\label{7}
Let $\SSDVR=\bbA_{\DVR}^d$ and let $A$ be a measurable subset of $\SSDVR_\infty$. Then $0\leq\mu_\SSDVR(A)\leq1$ and $\dim\mu_\SSDVR(A)\leq0$.
\end{lemma}
\begin{proof}
Suppose first that $A$ is stable of level $n$. We have $\pi_n(A)\subset\SSDVR_n=\bbA_\f^{dn}$, giving immediately that $\dim A=\dim\pi_n(A)-nd\leq nd-nd=0$, i.e., $\dim\mu_\SSDVR(A)\leq0$. Also $0\leq[\pi_n(A)]\leq\bbL^{dn}$, hence $0\leq\mu_\SSDVR(A)\leq1$.

In the general case there are, by definition, stable subsets $A_i$ such that $\mu_\SSDVR(A)=\lim_{i\to\infty}\mu_\SSDVR(A_i)$, i.e., $\mu_\SSDVR(A)=(\mu_\SSDVR(A_i))_{i\in\bbN}$. Since the equalities holds for stable sets it follows by definition that $0=(0)_{i\in\bbN}\leq\mu_\SSDVR(A)\leq(1)_{i\in\bbN}=1$. The same reasoning goes for the statement about dimension.
\end{proof}

\begin{lemma}\label{6}
Let $\SSDVR=\bbA_{\DVR}^d$ and let $A_i$ be measurable subsets of $\SSDVR_\infty$. If $e_i\to\infty$ as $i\to\infty$, then the sum $\sum_{i\in\bbN}\mu_\SSDVR(A_i)\bbL^{-e_i}$ is convergent.
\end{lemma}
\begin{proof}
By Lemma \ref{7}, $0\leq\mu_\SSDVR(A_i)\bbL^{-e_i}\leq\bbL^{-e_i}$. By Proposition \ref{MC-34} the sum $\sum_{i\in\bbN}\bbL^{-e_i}$ is convergent if and only if $\dim\bbL^{-e_i}\to-\infty$, i.e., $e_i\to\infty$. In this case it follows from Lemma 2.17 of \cite{MF} that $\sum_{i\in\bbN}\mu_\SSDVR(A_i)\bbL^{-e_i}$ is convergent.
\end{proof}

Let $\SSDVR^1=\bb{A}^1_{\DVR}$ and let $\SSDVR^1_\infty$ be its arc space. We have a function $\ord\colon\SSDVR^1_\infty\to\bbN\cup\{\infty\}$, mapping $x$ to the biggest power of the uniformizer dividing $x$, with the properties that $\ord ab=\ord a+\ord b$ and $\ord(a+b)\geq\min\{\ord a,\ord b\}$. We continue to write $\SSDVR$ for $\bbA^d_{\DVR}$. If $f\in\DVR[X_1\dotsk X_d]$ it defines a function $\SSDVR_\infty\to\SSDVR^1_\infty$. So for $n\in\bbN$ we may consider the subset $\{\ord f\geq n\}:=\set{(a_1\dotsk a_d)\in\SSDVR_\infty}{\ord f(a_1\dotsk a_d)\geq n}$. When $\DVR=\bbZp$ this set has the property that $\{\ord f\geq n\}\cap\bbZp^d=\set{(a_1\dotsk a_d)\in\bbZp^d}{\ord_pf(a_1\dotsk a_d)\geq n}$.
\begin{lemma}\label{leuven3}
Let $\SSDVR=\bbA^d_{\DVR}$. The subset $\{\ord f\geq n\}\subset\SSDVR_\infty$ is stable of level $n$ and $\mu_\SSDVR(\{\ord f\geq n\})=[\pi_n(\{\ord f\geq n\})]\bbL^{-dn}$.
\end{lemma}
\begin{proof}
To simplify the notation, we prove this for the special case when $\DVR=\W(\f)$; the general case is similar. When we view $\SSDVR_\infty$ as the $\overline{\f}$-points on the ring scheme $\W^d$, we see that $\{\ord f\geq n\}$ is actually the $\overline{\f}$-points on a closed subscheme. For let $\f[X_{i0},X_{i1}\dotsk X_{iN},\dots]_{i=1}^d$ represent $\W^d$. Let $f_0,f_1\dots$ be the universal polynomials defining $f$ in the Witt vectors, i.e., $f_n\in\f[X_{i0}\dotsk X_{in}]_{i=1}^d$ and if $R$ is any $\f$-algebra and $r_i=(r_{i0},r_{i1},\dots)\in\W(R)$ for $i=1\dotsk d$, then $f(r_1\dotsk r_d)=(f_0(r_1\dotsk r_d),f_1(r_1\dotsk r_d),\dots)\in\W(R)$. (We write $f_n(r_1\dotsk r_d)$ for $f_n(r_{10}\dotsk r_{d0};\dots; r_{1n}\dotsk r_{dn})\in R$.) We then see that $\{\ord f\geq n\}$ is identified with 
$$\set{(x_1\dotsk x_d)\in\W(\overline{\f})^d}{f(x_1\dotsk x_d)\equiv0\pmod{\V^n}}\subset\SSDVR_\infty,$$
i.e.,
$$\set{(x_1\dotsk x_d)\in\W(\overline{\f})^d}{f_i(x_1\dotsk x_d)=0\text{ for }i=0\dotsk  n-1}\subset\SSDVR_\infty,$$
This in turn is the $\overline{\f}$-points on the close subscheme
$$\spec\frac{\f[X_{i0},X_{i1}\dotsk X_{iN},\dots]_{i=1}^d}{(f_0\dotsk f_{n-1})}\subset\W^d.$$

Now $\pi_m(\SSDVR_\infty)=\W^d_m(\overline{\f})$. Hence, for $m\geq n$, we see that $\pi_{m}(\{\ord f\geq n\})$ is the $\overline{\f}$-points on the $\f$-scheme
$$\spec\frac{\f[X_{i0}\dotsk X_{i,m-1}]_{i=1}^d}{(f_0\dotsk f_{n-1})}.$$
In what follows we identify $\pi_{m}(\{\ord f\geq n\})$ with its underlying scheme. We then see that $\pi_m(\{\ord f\geq n\})=\pi_{n}(\{\ord f\geq n\})\times_{\f}\bbA_\f^{d(m-n)}$. The result follows.
\end{proof} 

Next consider the function $\underline{a}\mapsto\bbL^{-\ord f(\underline{a})}\colon\SSDVR_\infty\to\KOW$. For $a\in\W(\overline{\f})$ we write $\abs{a}:=\bbL^{-\ord a}$ and we want to compute the integral $\int_{\SSDVR_\infty}\abs{f}d\mu_\SSDVR$. The following proposition shows that the integral exists.
\begin{proposition}\label{11}
Let $\SSDVR=\bbA^d_{\DVR}$. Let $A$ be a measurable subset of $\SSDVR_\infty$, and $f\in\DVR[X_1\dotsk X_d]$. The integral $\int_A\abs{f}d\mu_\SSDVR=\int_A\bbL^{-\ord f}d\mu_\SSDVR$ exists. Moreover, when $\SSDVR=\bbA_{\bbZp}^d$ we have, for $q$ any power of $p$, $\counting_q\int_A\abs{f}d\mu_\SSDVR=\int_{A\cap\W(\bbFq)^d}\abs{f}_pd\muH$.
\end{proposition}
\begin{proof}
By definition the integral equals $\mu_\SSDVR(f=0)\cdot0+\sum_{i\in\bbN}\mu_\SSDVR\bigl(A\cap\{\ord f=i\}\bigr)\bbL^{-i}$. By Lemma \ref{leuven3}, $\{\ord f=i\}$ is stable, hence $A\cap\{\ord f=i\}$ is measurable. The integral therefore exists by Lemma \ref{6}. Next, by Proposition \ref{MC39}, $\counting_q\int_A\bbL^{-\ord f}d\mu_\SSDVR=\int_{A\cap\W(\bbFq)^d}\abs{f}_pd\muH$.
\end{proof}

The primary purpose of this paper is to show that the integral of the absolute value of a certain polynomial in many variables is a rational function in $\bbL$, with coefficients in $\bb{Z}$. For this we begin with some lemmas about general integrals of this kind of functions.

\begin{lemma}\label{17}
Let $A=\bigcup_{i\in\bbN}A_i$ be a disjoint union of stable subsets and suppose that $\sum_{i\in\bbN}\mu_\SSDVR(A_i)$ is convergent (so that $A$ is measurable). Then for any $f\in\DVR[X_1\dotsk X_d]$, we have $\int_A\abs{f}d\mu_\SSDVR=\sum_{i\in\bbN}\int_{A_i}\abs{f}d\mu_\SSDVR$.
\end{lemma}
\begin{proof}
By Proposition \ref{11} the integral exists. Since $A_i\cap\{\ord f=m\}$ is stable, and $A_i\cap\{\ord f=m\}\subset A_i$, it follows from Lemma 3.8 of \cite{MF} that the sum $\sum_{i\in\bbN}\mu_\SSDVR\bigl(A_i\cap\{\ord f=m\}\bigr)$ is convergent. Hence, since the union $A\cap\{\ord f=m\}=\bigcup_{i\in\bbN}A_i\cap\{\ord f=m\}$ is disjoint it follows from Proposition 3.7 of \emph{loc.cit.} that $\mu_\SSDVR\bigl(A\cap\{\ord f=m\}\bigr)=\sum_{i\in\bbN}\mu_\SSDVR\bigl(A_i\cap\{\ord f=m\}\bigr)$. We may therefore write
\begin{align*}
\int_{A}\abs{f}d\mu_\SSDVR&=\sum_{m\in\bbN}\mu_\SSDVR\bigl(A\cap\{\ord f=m\}\bigr)\bbL^{-m}\\
&=\sum_{m\in\bbN}\sum_{i\in\bbN}\mu_\SSDVR\bigl(A_i\cap\{\ord f=m\}\bigr)\bbL^{-m}.
\end{align*}
Because of Lemma \ref{6}, if we do the above summation over an enumeration of $\bbN^2$, it is convergent. Hence by Lemma 2.19 of \emph{loc.cit} it equals
\begin{equation*}
\sum_{i\in\bbN}\sum_{m\in\bbN}\mu_\SSDVR\bigl(A_i\cap\{\ord f=m\}\bigr)\bbL^{-m}=\sum_{i\in\bbN}\int_{A_i}\abs{f}d\mu_\SSDVR.\qedhere
\end{equation*}
\end{proof}

Let $f_1\dotsk f_r\in\DVR[X_1\dotsk X_d]$. For $\alpha=(\alpha_1\dotsk\alpha_r)\in\bbN^r$ we write $\{\ord f_i=\alpha_i\}_{i=1}^r$ for the subset $\set{a\in\SSDVR_\infty}{\ord f_i(a)=\alpha_i}_{i=1}^r\subset\SSDVR_\infty$.

\begin{lemma}\label{8}
Let $\SSDVR=\bbA^d_{\DVR}$. For $I\subset(\bbN\cup\{\infty\})^d$ (finite or infinite), let $U_I:=\bigcup_{\alpha\in I}\{\ord X_i=\alpha_i\}_{i=1}^d$. Let $f\in\DVR[X_1\dotsk X_d]$. Then $\int_{U_I}\abs{f}d\mu_\SSDVR=\sum_{\alpha\in I}\int_{\{\ord X_i=\alpha_i\}}\abs{f}d\mu_\SSDVR$. (In particular the integral exists.)
\end{lemma}
\begin{proof}
We show that $U_I$ is measurable, the result then follows from Lemma \ref{17}. Since the union $U_{I}=\bigcup_{\alpha\in I}\{\ord X_i=\alpha_i\}_{i=1}^d$ is disjoint it suffices, by Proposition 3.7 of \cite{MF}, to prove convergence of the sum
$$\sum_{\alpha\in I}\mu_\SSDVR(\{\ord X_i=\alpha_i\}_{i=1}^d).$$

Let $N$ be a large integer. We have $\pi_{N}(\{\ord X_i=\alpha_i\}_{i=1}^d)\subset\pi_{N}(\{\ord X_i\geq\alpha_i\}_{i=1}^d)$. The underlying scheme of this latter set is
$$\spec\frac{\f[X_{i0}\dotsk X_{iN}]_{i=1}^d}{(X_{i0}\dotsk X_{i,\alpha_i-1})_{i=1}^d}=\spec\f[X_{i\alpha_i}\dotsk X_{iN}]_{i=1}^d.$$
Hence $[\pi_{N}(\{\ord X_i=\alpha_i\}_{i=1}^d)]\leq\bbL^{N d-\sum_{i=1}^d\alpha_i}$ and consequently $\mu_\SSDVR(\{\ord X_i=\alpha_i\}_{i=1}^d)\leq\bbL^{-\sum_{i=1}^d\alpha_i}$. Now, when $I$ is infinite we see that as $\alpha$ varies over $I$, $\max\{\alpha_i\}_{i=1}^d\to\infty$, hence that $-\sum_{i=1}^d\alpha_i\to-\infty$. So by Proposition \ref{MC-34}, $\sum_{\alpha\in I}\bbL^{-\sum_{i=1}^d\alpha_i}$ is convergent, hence by Lemma 2.17 of \emph{loc.cit}, $\sum_{\alpha\in I}\mu_\SSDVR(\{\ord X_i=\alpha_i\}_{i=1}^d)$ is convergent.
\end{proof}

\section{The motivic integral of a polynomial in one variable}\label{24}
Let $\DVR$ be a complete discrete valuation ring with residue field $\f$, constructed as in Subsection \ref{MC-103}. In the preceding section we proved the existence of the integral $\int_{\SSDVR_\infty}\abs{f}d\mu_{\SSDVR}\in\KOW$, where $\SSDVR=\bbA^d_{\DVR}$, and $f$ is a polynomial in $d$ variables with coefficients in $\DVR$. In this section we compute this integral more explicitly, in the case when $d=1$. This type of computation is standard in the theory of motivic integration. We include this section anyway, in order to fix notation for the later sections, and also to give a simple example.

In this section we let $\SSDVR=\bbA_{\DVR}^1$. We use the notation of the preceding section. In particular, let $f_i\in\f[X_0\dotsk X_i]$ be the universal polynomials that define the function $x\mapsto f(x)\colon\SSDVR_\infty\to\SSDVR_\infty$. More precisely, the $f_i$ are such that if $\DVR=\W(\f)$ and
$$X:=(X_0,X_1\dotsk X_i,\dots)\in\W(\f[X_0,X_1\dotsk X_i,\dots])$$
then $$f(X)=\bigl((f_0(X_0),f_1(X_0,X_1)\dotsk f_i(X_0\dotsk X_i),\dots\bigr)\in\W(\f[X_0,X_1\dotsk X_i,\dots])$$
If $\DVR=\psr{\f}{t}$ and $X:=\sum_{i\geq0}X_it^i\in\psr{\f[X_0\dotsk X_i,\dots]}{t}$ then $f(X)=\sum_{i\geq0}f_i(X_0\dotsk X_{i})t^i$.

Let $x=(x_0,x_1\dotsk x_i,\dots)\in\SSDVR_\infty$ (recall that we write $(x_0,x_1,\dots)$ for $\sum_{i\geq0}x_it^i$ in order to get a uniform notation). Write $f_i(x)$ for $f_i(x_0\dotsk x_i)$. That the inequality $\ord f(x)\geq n$ holds is then equivalent to $f_0(x)=f_1(x)=\dots=f_{n-1}(x)=0$. If $m\geq n$ we then see that $\pi_m(x)=(x_0\dotsk x_{m-1})\in\SSDVR_m(\overline{\f})$ belongs to $\pi_m\bigl(\{\ord f\geq n\}\bigr)$ if and only if $f_0(x)=f_1(x)=\dots=f_{n-1}(x)=0$. Therefore, $\pi_m\bigl(\{\ord f\geq n\}\bigr)$ equals the $\overline{\f}$-points on the $\f$-scheme $\spec\tfrac{\f[X_0\dotsk X_{m-1}]}{(f_0\dotsk f_{n-1})}$, and $\pi_m^{-1}\pi_m\bigl(\{\ord f\geq n\}\bigr)=\{\ord f\geq n\}$. In what follows we will identify $\pi_m\bigl(\{\ord f\geq n\}\bigr)$ with its underlying scheme.

\begin{proposition}[Motivic Newton's Lemma]\label{MC-105}
Let $f\in\DVR[X]$ and assume that $f_0$ is non-constant and separable. Consider the subset $\{\ord f\geq n\}\subset\SSDVR_\infty$, where $n\geq1$. We have an isomorphism of $\f$-schemes $\pi_n\bigl(\{\ord f\geq n\}\bigr)\to\pi_1\bigl(\{\ord f\geq n\}\bigr)$. In particular, $\mu_\SSDVR\bigl(\{\ord f\geq n\}\bigr)=\Bigl[\spec\tfrac{\f[X_0]}{(f_0)}\Bigr]\bbL^{-n}$.
\end{proposition}
\begin{proof}
Let
\begin{align*}
R_i:=\frac{\field[X_0,\dots,X_{i-1}]}{(f_0,\dots,f_{i-1})} && i\geq1.
\end{align*}
Then $\pi_i\bigl(\{\ord f\geq n\}\bigr)=\spec R_i$ for $i=1,\dots,n$ and we want to prove that $R_n\simeq R_1$. We do this by proving that the canonical homomorphism $R_i\rightarrow R_{i+1}$ is an isomorphism for every $i\geq1$.

To simplify notation, we do this only in the mixed characteristic case. Let $x:=(X_0,\dots,X_i)\in\W_{i+1}\bigl(\field[X_0,\dots,X_i]\bigr)$ and $\tilde{x}:=(X_0,\dots,X_{i-1},0)$. From \eqref{41} it follows that $x=\tilde{x}+\V^i\R(X_i)$, so we may Taylor expand to get
\begin{equation}\label{eq:43}
\begin{split}
f(x)=&f\bigl(\tilde{x}+\V^i\R(X_i)\bigr)\\
=&f(\tilde{x})+\frac{\partial f}{\partial X}(\tilde{x})\cdot\V^i\R(X_i)+O\bigl(\V^i\R(X_i)\bigr)^2\in\W_{i+1}\bigl(\f[X_0\dotsk X_i]\bigr).
\end{split}
\end{equation}
Here $f(\tilde{x})=(f_0,\dots,f_{i-1},q)$, where $q$ is a polynomial in $\f[X_0,\dots,X_{i-1}]$. Moreover, since $\pi_1f=f_0$, it follows that if $\tfrac{\partial f}{\partial X}(x)=(f^*_0,\dots,f^*_i)$ then $f_0^*=\tfrac{\partial f_0}{\partial X_0}$. Hence $\tfrac{\partial f}{\partial X}(\tilde{x})=(\tfrac{\partial f_0}{\partial X_0},\dots)$. Finally by Proposition \ref{prop:40} $\bigl(\V^i\R(X_i)\bigr)^2=F^i\V^{2i}\bigl(\R(X_i)\bigr)=0\in\W_{i+1}\bigl(\f[X_1\dotsk X_i]\bigr)$. Hence if we write explicitly we see that the right hand side of \eqref{eq:43} is
\begin{align*}
(f_0,\dots,f_{i-1},q)+(\tfrac{\partial f_0}{\partial X_0},\dots)\cdot(0,\dots,0,X_i)
=&(f_0,\dots,f_{i-1},q)+(0,\dots,0,(\tfrac{\partial f_0}{\partial X_0})^{p^i}X_i)\\
=&(f_0,\dots,f_{i-1},q+(\tfrac{\partial f_0}{\partial X_0})^{p^i}X_i).
\end{align*}
On the other hand, by definition, $f(x)=(f_0,\dots,f_i)\in\W_{i+1}\bigl(\field[X_0,\dots,X_i]\bigr)$ hence we get the identity
\begin{equation}\label{eq:44}
f_i(X_0,\dots,X_i)=q(X_0,\dots,X_{i-1})+(\tfrac{\partial f_0}{\partial X_0})^{p^i}\cdot X_i
\end{equation}
in $\field[X_0,\dots,X_i]$.

We shall also use the hypothesis that $f$ is separable modulo $\V$. This means that $\tfrac{\partial f_0}{\partial X_0}$ is invertible in $R_1$. Let $(\tfrac{\partial f_0}{\partial X_0})^{-1}$ and $h\in\f[X_0]$ be such that $\tfrac{\partial f_0}{\partial X_0}\cdot(\tfrac{\partial f_0}{\partial X_0})^{-1}=1+hf_0$ in $\field[X_0]$.

We now prove that $R_i\rightarrow R_{i+1}$ is injective. Let $g\in\f[X_1\dotsk X_{i-1}]$. We have to prove that if $\overline{g}=0\in R_{i+1}$ then $\overline{g}=0\in R_i$. So suppose that $g=h_0f_0+\dots+h_if_i\in\f[X_0\dotsk X_i]$, where $h_j\in\f[X_1\dotsk X_i]$. By \eqref{eq:44} this gives
\begin{equation*}
g=h_0f_0+\dots+h_{i-1}f_{i-1}+h_i\cdot\bigl(q+\tfrac{\partial f_0}{\partial X_0}^{p^i}\cdot X_i\bigr)\in\f[X_0\dotsk X_i].
\end{equation*}
Substituting $-q\cdot(\tfrac{\partial f_0}{\partial X_0})^{-p^i}$ for $X_i$ then gives
\begin{equation*}
g=h^*_0f_0+\dots+h_{i-1}^*f_{i-1}+h_i^*\cdot(q-q+h^*f_0)\in\f[X_0\dotsk X_{i-1}],
\end{equation*}
where the $h_j^*$ and $h^*$ are polynomials in $\f[X_0,\dots,X_{i-1}]$. Hence $\overline{g}=0\in R_i$ and consequently $R_i\rightarrow R_{i+1}$ is injective.

Finally we prove that $R_i\rightarrow R_{i+1}$ is surjective. It suffices to show that $\overline{X}_i$ is in the image of $R_i$. Working in $R_{i+1}$, \eqref{eq:44} becomes
\begin{equation}\label{40}
0=\overline{q}+\overline{\tfrac{\partial f_0}{\partial X_0}}^{p^i}\cdot\overline{X_i}.
\end{equation}
Now $\overline{\tfrac{\partial f_0}{\partial X_0}}\in R_1$ is invertible, and $R_1\neq0$ by assumption. Since $R_1\to R_{i+1}$ is injective it hence follows that the image of $\overline{\tfrac{\partial f_0}{\partial X_0}}$ in $R_{i+1}$ is invertible, consequently we can write \eqref{40} as
\begin{equation*}
\overline{X_i}=-\overline{q}\cdot\overline{\tfrac{\partial f_0}{\partial X_0}}^{-p^i}\in R_{i+1}.
\end{equation*}
The right hand side involves only the variables $X_0,\dots,X_{i-1}$, hence is in the image of $R_i$.
\end{proof}

\begin{proposition}\label{MC-21}
Let $f\in\DVR[X]$ and assume that $f_0\in\f[X_0]$ is separable and non-constant. Then
$$\int\abs{f}d\mu_\SSDVR=1-[\spec\f[X_0]/(f_0)]\frac{1}{\bbL+1}\in\KOW.$$
\end{proposition}
\begin{proof}
By definition we have
\begin{equation*}
\int_{\SSDVR_\infty}\abs{f}d\mu_{\SSDVR}=\sum_{m\geq0}\bb{L}^{-m}\mu_{\SSDVR}\{\ord f=m\}.
\end{equation*}
Since $\{\ord f=m\}=\{\ord f\geq m\}\setminus\{\ord f\geq m+1\}$  we have
\begin{equation*}
\mu_{\SSDVR}\{\ord f=m\}=\Bigl[\spec\tfrac{\f[X_0]}{(f_0)}\Bigr]\cdot(\bb{L}^{-m}-\bb{L}^{-(m+1)})
\end{equation*}
for $m\geq1$. For $m=0$ we have $\mu_{\SSDVR}(\ord f=0)=\mu_{\SSDVR}(\SSDVR_\infty\setminus\{\ord f\geq1\}) =1-\Bigl[\spec\tfrac{\f[X_0]}{(f_0)}\Bigr]\bb{L}^{-1}$. Therefore, using Proposition \ref{MC-34},
\begin{align*}
\int_{\SSDVR_\infty}\abs{f}d\mu_{\SSDVR}=&1+\Bigl[\spec\tfrac{\f[X_0]}{(f_0)}\Bigr]\cdot\biggl(-\bb{L}^{-1}+\sum_{m\geq1}\bb{L}^{-m}\bigl(\bb{L}^{-m}-\bb{L}^{-(m+1)}\bigr)\biggr)\\
=&1-\Bigl[\spec\tfrac{\f[X_0]}{(f_0)}\Bigr]\cdot\frac{1}{\bb{L}+1}.\qedhere
\end{align*}
\end{proof}
\begin{example}\label{MC9}
We look at the case when $\DVR=\bbZp$: If $f=aX+b$, where $a\in\bbZp^\times$, then $\tfrac{\f[X_0]}{(f_0)}=\f$. Since $[\spec\f]=1$ we have $\int_{\SSDVR_\infty}\abs{aX+b}d\mu_{\SSDVR}=\frac{\bb{L}}{\bb{L}+1}$, showing in particular that if $q$ is a power of $p$ then  $\int_{\W(\bbFq)}\abs{aX+b}dX=\frac{q}{q+1}$

More generally, assume that $f$ is such that $f_0$ is irreducible of degree $d$. Then $\spec\f[X_0]/(f_0)\simeq\bbF_{p^d}$, hence $\int_{\SSDVR_\infty}\abs{f}d\mu_{\SSDVR}=1-\tfrac{[\spec\bb{F}_{p^d}]}{\bbL+1}$. Applying $\counting_{q}$ for different powers of $p$ shows that
$$\int_{\W(\bbF_{q})}\abs{f}_pd\muH=\begin{cases}
1-d/(q+1) & q=p^i\text{ where }d\mid i\\
1 & q=p^i \text{ where } d\nmid i
\end{cases}.$$
\end{example}

\section{Change of variables}\label{MC8}
We prove three theorems about manipulation of these kind of integrals. Recall that we use $\DVR$ to denote a complete discrete valuation ring with perfect residue field $\f$.
\subsection{Linear change of variables}
A linear change of variables is easy to do also in the motivic case:
\begin{proposition}\label{MC-19}
Let $\SSDVR=\bbA^d_{\DVR}$ and let $a_{ij}\in\DVR$ be such that the determinant of $M=(a_{ij})$ is in $\DVR^\times$. Given $f\in\bbZp[X_1\dotsk X_d]$, define $g(X_1\dotsk X_d):=f\bigl((X_1\dotsk X_d)M\bigr)$. Then $\int_{\SSDVR_\infty}\abs{f}d\mu_\SSDVR=\int_{\SSDVR_\infty}\abs{g}d\mu_\SSDVR$.
\end{proposition}
\begin{proof}
We first prove that $\mu_\SSDVR\{\ord f\geq n\}=\mu_\SSDVR\{\ord g\geq n\}$ for every $n$. We have a map
$$\{\ord f\geq n\}\to\{\ord g\geq n\},$$
given by $(x_1\dotsk x_d)\mapsto(x_1\dotsk x_d)M^{-1}$. This is a bijection, for it is well defined since $g(\underline{x}M^{-1})=f(\underline{x}M^{-1}M)=f(\underline{x})\equiv0\pmod{\V^n}$, and it has a well defined inverse $\underline{x}\mapsto\underline{x}M$. Therefore $\pi_n\{\ord f\geq n\}$ and $\pi_n\{\ord g\geq n\}$ are isomorphic (viewed as subschemes of $\SSDVR_n$), consequently $[\pi_n\{\ord g\geq n\}]=[\pi_n\{\ord f\geq n\}]$.


It follows that $\mu_\SSDVR\{\ord f=n\}=\mu_\SSDVR\{\ord g=n\}$ for every $n$, hence that the integrals are equal.
\end{proof}

\subsection{Separation of variables}
Throughout this subsection, let $\SSDVR:=\bbA_{\DVR}^d$ and $\mathcal{Y}:=\bbA_{\DVR}^e$. Moreover, let $\mathcal{Z}:=\bbA_{\DVR}^{d+e}=\SSDVR\times_{\DVR}\mathcal{Y}$. We may then identify $(\mathcal{Z})_\infty$ with  $\SSDVR_\infty\times\mathcal{Y}_\infty$. Our aim is to show the separation of variables result, Theorem \ref{MC-15}. We do this using two partial results, which we state as the following two lemmas:

\begin{lemma}\label{30}
If $A\subset\SSDVR_\infty$ and $B\subset\mathcal{Y}_\infty$ are stable, then $A\times B\subset\mathcal{Z}_\infty$ is stable, and $\mu_{\mathcal{Z}}(A\times B)=\mu_\SSDVR(A)\mu_{\mathcal{Y}}(B)$.
\end{lemma}
\begin{proof}
Since $A$ and $B$ are stable there is an integer $n$ with the property that there are a finite number of $\f$-varieties $V_i$ such that $\pi_n(A)=\bigcup_{i}V_i(\overline{\f})$, and a finite number of $\f$-varieties $U_i$ such that $\pi_n(A)=\bigcup_{i}U_i(\overline{\f})$. We have
\begin{align*}
\pi_n(A\times B)=&\pi_n(A)\times\pi_n(B)\\
=&\bigcup_{i}V_i(\overline{\f})\times\bigcup_j U_j(\overline{\f})\\
=&\bigcup_{i,j}V_i(\overline{\f})\times U_j(\overline{\f})\\
=&\bigcup_{i,j}(V_i\times_{\f} U_j)(\overline{\f}),
\end{align*}
hence $[\pi_n(A\times B)]=\sum_{i,j}[V_i\times_{\f} U_j]=\sum_{i,j}[V_i][U_j]=(\sum_i[V_i])(\sum_j[U_j])=[\pi_n(A)][\pi_n(B)]$. Therefore $\mu_{\SSDVR\times\mathcal{Y}}(A\times B)=[\pi_n(A\times B)]\bbL^{-n(d+e)}=([\pi_n(A)]\bbL^{-nd})([\pi_n(B)]\bbL^{-ne})=\mu_\SSDVR(A)\mu_{\mathcal{Y}}(B)$.
\end{proof}

\begin{lemma}
If $A\subset\SSDVR_\infty$ and $B\subset\mathcal{Y}_\infty$ are measurable, then $A\times B\subset\mathcal{Z}_\infty$ is measurable, and $\mu_{\mathcal{Z}}(A\times B)=\mu_\SSDVR(A)\mu_{\mathcal{Y}}(B)$.
\end{lemma}
\begin{proof}
Let $A_m$ and $C_i^m$ be stable subsets of $\SSDVR_\infty$ such that $A\symdiff A_m\subset\bigcup_{i\in\bbN}C_i^m$. Let $u_m:=\sum_{i\in\bbN}\mu_\SSDVR(C_i^m)$ be convergent and $\lim_{m\to\infty}u_m=0$. Let $B_m$ and $D_i^m$ be stable subsets of $\mathcal{Y}_\infty$ such that $B\symdiff B_m\subset\bigcup_{i\in\bbN}D_i^m$, where $v_m:=\sum_{i\in\bbN}\mu_\SSDVR(D_i^m)$ is convergent and $\lim_{m\to\infty}v_m=0$. Then $(A\times B)\symdiff(A_m\times B_m)=(A\symdiff A_m)\times(B\symdiff B_m)\subset\bigcup_{i\in\bbN}C_i^m\times D_i^m$. By Lemma \ref{30}, $\mu_\mathcal{Z}(C_i^m\times D_i^m)=\mu_\SSDVR(C_i^m)\mu_\mathcal{Y}(D_i^m)$, hence by Lemma 2.21 of \cite{MF}, the sum $s_m:=\sum_{i\in\bbN}\mu_\mathcal{Z}(C_i^m\times D_i^m)$ is convergent, and $s_m\leq u_mv_m$. Since $u_m$ and $v_m$ tends to zero the same holds for $u_mv_m$ and consequently also for $s_m$. Hence, since $A_m$ and $B_m$ are stable,
\begin{equation*}
\mu_\mathcal{Z}(A\times B)=\lim_{m\to\infty}\mu_\mathcal{Z}(A_m\times B_m)=\lim_{m\to\infty}\mu_\SSDVR(A_m)\mu_\mathcal{Y}(B_m)=\mu_\SSDVR(A)\mu_{\mathcal{Y}}(B).\qedhere
\end{equation*}
\end{proof}

\begin{theorem}[Separation of variables]\label{MC-15}
Let $\SSDVR:=\bbA_{\DVR}^d$, $\mathcal{Y}:=\bbA_{\DVR}^e$ and $\mathcal{Z}=\bbA_{\DVR}^{d+e}$. If $A\subset\SSDVR_\infty$ and $B\subset\mathcal{Y}_\infty$ are measurable, and $f\in\DVR[X_1\dotsk X_d]$, $g\in\DVR[Y_1\dotsk Y_e]$, then $\int_{A\times B}\abs{fg}d\mu_{\mathcal{Z}}=\int_A\abs{f}d\mu_{\SSDVR}\cdot\int_B\abs{g}d\mu_{\mathcal{Y}}$.
\end{theorem}
\begin{proof}
By Proposition \ref{11} the integral is convergent:
\begin{equation*}
\int_{A\times B}\abs{fg}d\mu_{\mathcal{Z}}=\sum_{\xi\in\bbN}\mu_{\mathcal{Z}}((A\times B)\cap\{\ord fg=\xi\})\bbL^{-\xi}
\end{equation*}
Since $\{\ord fg=\xi\}=\bigcup_{\mu+\nu=\xi}\{\ord f=\mu\}\times\{\ord g=\nu\}$ we have 
$$(A\times B)\cap\{\ord fg=\xi\}=\bigcup_{\mu+\nu=\xi}\bigl(A\cap\{\ord f=\mu\}\bigr)\times\bigl(B\cap\{\ord g=\nu\}\bigr)$$
and since this is a disjoint union of measurable sets it follows from the previous lemma that
$$\int_{A\times B}\abs{fg}d\mu_{\mathcal{Z}}=\sum_{\xi\in\bbN}\biggl(\sum_{\mu+\nu=\xi}\mu_\SSDVR\bigl(A\cap\{\ord f=\mu\}\bigr)\cdot\mu_{\mathcal{Y}}\bigl(B\cap\{\ord g=\nu\}\bigr)\biggr)\bbL^{-\xi}$$
Since this sum is convergent, Lemma 2.20 of \cite{MF} says that we may rearrange it to obtain
\begin{equation*}
\Biggl(\sum_{\mu\in\bbN}\mu_\SSDVR\bigl(A\cap\{\ord f=\mu\}\bigr)\bbL^{-\mu}\Biggr)\Biggl(\sum_{\nu\in\bbN}\mu_{\mathcal{Y}}\bigl(B\cap\{\ord f=\nu\}\bigr)\bbL^{-\nu}\Biggr)=\int_A\abs{f}d\mu_{\SSDVR}\cdot\int_B\abs{g}d\mu_{\mathcal{Y}}\qedhere
\end{equation*}
\end{proof}

\subsection{Multiplication by the uniformizer}\label{MC-106}
In this subsection we prove a motivic version of the change of variables induced by multiplication by the uniformizer of the discrete valuation ring.

To simplify notation, we write $\f[X_{\bullet0}\dotsk X_{\bullet N}]$ for the polynomial ring $\f[X_{i0}\dotsk X_{iN}]_{i=1}^n$. As usual, we use $Q_N$ to denote the universal polynomials defining $Q\in\DVR[X_1\dotsk X_n]$, so that $Q_N\in\f[X_{\bullet 0}\dotsk X_{\bullet N}]$. (See the discussion in the proof of Lemma \ref{leuven3}, or in the beginning of Section \ref{24}.)
\begin{lemma}\label{101}
Let $\SSDVR=\bbA^n_{\DVR}$. Let $Q\in\DVR[X_1\dotsk X_n]$ be a form of degree $s$. If $\ord x_i>0$ for $i=1\dotsk n$, then $\ord Q\geq s$. Moreover for every $\xi\in\bbN$
$$\mu_{\SSDVR}(\{\ord Q>\xi+s,\ord x_i>0\}_{i=1}^n)=\bbL^{-n}\mu_{\SSDVR}(\{\ord Q>\xi\}).$$
\end{lemma}
\begin{proof}
For $N$ sufficiently large, $\pi_{N+1}(\ord \Delta>\xi+s,\ord x_i>0)$ is the spectrum of the algebra
\begin{equation*}
\frac{\f[X_{\bullet0}\dotsk X_{\bullet N}]}{\bigl(Q_0\dotsk Q_{\xi+s},X_{\bullet0}\bigr)}.
\end{equation*}
In the mixed characteristic case, it follows from Corollary \ref{co:45} that the class of this in $\KOR$ equals (since $[X]=[X_{red}]$ for any scheme $X$) the class of the spectrum of
\begin{equation*}
\frac{\f[X_{\bullet1}\dotsk X_{\bullet N}]}{\bigl(Q_0(X_{\bullet1})\dotsk Q_{\xi}(X_{\bullet1}\dotsk X_{\bullet\xi+1})\bigr)}.
\end{equation*}
In the equal characteristic case, this is straight forward to prove. Now, using the change of variables $X_{\bullet i}\mapsto X_{\bullet i-1}$, the spectrum of this algebra is $\pi_N(\ord Q>\xi)$. The result follows.
\end{proof}

\begin{theorem}\label{103}
Let $\SSDVR=\bbA^n_{\DVR}$, let $Q\in\DVR[X_1\dotsk X_n]$ be a form of degree $s$ and let $A=\{\ord x_i>0\}_{i=1}^n\subset\SSDVR_\infty$. Then $\int_A\abs{Q}d\mu_\SSDVR=\bbL^{-s-n}\int_{\SSDVR_\infty}\abs{Q}d\mu_\SSDVR$.
\end{theorem}
\begin{proof}
By the first part of Lemma \ref{101}, $\{\ord Q=\xi,\ord x_i>0\}_{i=1}^n=\emptyset$ for $\xi<s$, hence
\begin{equation*}
\int_{A}\abs{Q}d\mu_\SSDVR =\sum_{\xi\geq0}\mu_\SSDVR\{\ord Q=\xi+s,\ord x_i>0\}\bbL^{-(\xi+s)}.
\end{equation*}
Using the second part of the lemma it follows that this equals
\begin{equation*}
\sum_{\xi\geq0}\bbL^{-n}\mu_\SSDVR(\ord Q=\xi)\bbL^{-(\xi+s)}=\bbL^{-s-n}\int_{\SSDVR_\infty}\abs{Q}d\mu_\SSDVR.\qedhere
\end{equation*}
\end{proof}

\section{The motivic integral of the absolute value of a product of linear forms}\label{MC3}

As before, we let $\DVR$ be a complete discrete valuation ring with residue field $\f$, of one of the types described in Subsection \ref{MC-103}. We define, for any $n\in\bbN$, $\SSDVR^n=\bbA^n_{\DVR}$, and we let $\SSDVR_\infty^n$ be its space of arcs.

The main result of this section is Theorem \ref{mainMC}, and also Theorem \ref{leuven1}. In Theorem \ref{mainMC} we give a recursive method to compute $I=\int_{\SSDVR^n_\infty}\abs{\prod\ell_i}d\mu_{\SSDVR^n}$, where $\ell_i\in\bbZ[X_1\dotsk X_n]$ are linear forms. When $\DVR=\psr{\f}{t}$, these forms are arbitrary; when $\DVR=\W(\f)$ we need the forms to be of a rather special type. The restriction in the second case is taken care of in Theorem \ref{leuven1}, when we give a recursion that works for general forms in case $\DVR=\W(\f)$, provided that the characteristic of $\f$ is sufficiently large.

When applicable, these theorems also give a function $f\in\bbZ(T)$ such that $I=f(\bbL)$. If we let $\DVR=\bbZp$ this also gives a motivic explanation to the phenomenon discussed in the introduction. For by applying $\counting_q$ to $I$, for different powers $q$ of $p$, we get $\int_{\W(\bbFq)}\abs{\prod\ell_i}d\muH=f(q)$.

\begin{remark}
As mentioned in the introduction, it is not true in general that the motivic integral of the absolute value of a polynomial is equal to $f(\bbL)$, with $f\in\bbZ(T)$. This can be seen from the integral $\int_{\SSDVR^1_\infty}\abs{x^2+1}d\mu_{\SSDVR^1}$: When $\DVR=\bbZp$ with $p\equiv 3\mod 4$, then by Example \ref{MC9} this integral is equal to $1-[\spec\bb{F}_{p^2}]/(\bbL+1)$, and by applying the point counting homomorphism for different powers of $p$ we see that this cannot be equal to $f(\bbL)$ for $f\in\bbZ(T)$.
\end{remark}

Let us mention one remaining question about these integrals: From the computations performed in Theorem \ref{mainMC} and \ref{leuven1} it is clear that $f\in\bb{Z}(T)$, the rational function with the property that $\int_{\SSDVR^n_\infty}\abs{\prod\ell_i}d\mu_{\SSDVR^n}=f(\bbL)$, is independent of $\DVR$, provided that we choose $\DVR$ among the rings $\set{\W(\f)}{p\text{ sufficiently large}}\cup\{\psr{\f}{t}\}$. In particular, we have $\int_{\bbZp^n}\abs{\prod\ell_i}d\muH=f(p)$ for $p$ big enough. It would be desirable to have a motivic explanation also for this fact. This can probably be achieved using the theory of motivic integration developed in \cite{MR2403394}. Alternatively, we indicate in the following remark how the problem could be handled using geometric motivic integration.
\begin{remark}
By Theorem 6.1 of \cite{MR1905328}, if $\DVR=\psr{\bbQ}{t}$, if $\mathcal{Y}=\{P=0\}\subset\SSDVR^n$ where $P$ is a polynomial, and if $J(T)=\sum_{i\geq0}[\mathcal{Y}_{i+1}]T^i\in\psr{\KOLunmarked_{\bbQ}}{T}$, then the following holds: Firstly, $J(T)$ is rational, with denominator consisting of factors of the form $1-\bbL^aT^b$. Moreover, if we choose representatives for the coefficients of $J(T)$, defined over $\bbZ$, and then count $\bbFp$-points on them, then for $p$ sufficiently large we get the power series $J_p(T)=\sum_{i\geq0}\abs{\set{\underline{x}\in(\bbZ/(p^{i+1}))^n}{P(\underline{x})=0}}T^i$.

The process of choosing representatives for elements of $\KOLunmarked_{\bbQ}$, and then counting $\bbFp$-points on them for all $p$, defines a homomorphism $\counting\colon\KOLunmarked_{\bbQ}\to\prod_{p}\bbQ/{\sim}$, where $(a_p)$ and $(b_p)$ are equivalent if $a_p=b_p$ for almost all $p$. (The filter product with respect to the Fr\'echet filter.) Since $\int_{\bbZp^n}\abs{P}d\muH=1+p^{-n-1}(1-p)J_p(p^{-1-n})$, one could define the motivic integral of $P$ by first computing $J$ as a rational function, and then define the integral to be $I=1+\bbL^{-n-1}(1-\bbL)J(\bbL^{-1-n})\in\KOLunmarked_{\bbQ}[(1-\bbL^i)^{-1}]_{i\geq1}$. This integral then has the property that $\counting I=\bigl(\int_{\bbZp^n}\abs{P}d\muH\bigr)_p\in\prod_{p}\bbQ/{\sim}$. For example, let $P=X^2+1$, and let $m=[\spec\bbQ[X]/(P(X))]$. Using Theorem \ref{MC-105} one sees that $J(T)=m/(1-T)$, so the integral of $P$ is $1-m/(\bbL+1)$. Hence, for $p$ sufficiently large the value of $\int_{\bbZp}\abs{P}d\muH$ is $1$ if $p\equiv1\mod 4$ and $(p-1)/(p+1)$ if $p\equiv3\mod 4$ (a result that of course is true for all $p$). Probably the method used to prove Theorem \ref{mainMC} can be used also to compute $J(T)$ when $P$ is a product of linear forms, showing that the integral equals $f(\bbL)$, hence that $\int_{\bbZp^n}\abs{\prod\ell_i}d\muH=f(p)$ for $p$ big enough.
\end{remark}

We now give an example showing that $f$ is not independent of $p$ for all $p$, only for $p$ sufficiently large. For fix a prime $l$. Let $p$ be any prime different from $l$, and let $\SSDVR^2=\bbA^2_{\bbZp}$. Then, using Proposition \ref{MC-19}, Proposition \ref{MC-15}, and Proposition \ref{MC-21}, we see that 
$$\int_{\SSDVR^2_\infty}\abs{(x_1+x_2)(x_1-(l-1)x_2)}d\mu_{\SSDVR^2}=\int_{\SSDVR^2_\infty}\abs{y_1y_2}d\mu_{\SSDVR^2}=\Bigl(\int_{\SSDVR^1_\infty}\abs{y}d\mu_{\SSDVR^1}\Bigr)^2=\tfrac{\bbL^2}{(\bbL+1)^2}.$$
If this formula were true for $p=l$, then it would follow that $\int_{\bbZp^2}\abs{(x_1+x_2)(x_1-(p-1)x_2)}dx_1dx_2=\tfrac{p^2}{(p+1)^2}$, contradicting the following example.
\begin{example}
Consider the linear mapping $(x_1,x_2)\mapsto(x_1+x_2,x_1-(p-1)x_2)\colon\bbZp^2\to\bbZp^2$. It is easy to check that it is injective, that its image is $\bigcup_{a=0}^{p-1}(a+p\bbZp)^2$, and that its Jacobian is constant of absolute value $1/p$. Hence
$$\int_{\bbZp^2}\abs{(x_1+x_2)(x_1-(p-1)x_2)}dx_1dx_2=p\sum_{a=0}^{p-1}\int_{(a+p\bbZp)^2}\abs{y_1y_2}dy_1dy_2.$$
Now if $a\neq0$, $\int_{(a+p\bbZp)^2}\abs{y_1y_2}dy_1dy_2=(\int_{a+p\bbZp}dy)^2=1/p^2$, whereas $\int_{(p\bbZp)^2}\abs{y_1y_2}dy_1dy_2=(\int_{p\bbZp}\abs{y}dy)^2=(\int_{\bbZp}\abs{y}dy-\int_{\bbZp^\times}dy)^2=(p/(p+1)-(p-1)/p)^2=1/(p(p+1))^2$, hence $$\int_{\bbZp^2}\abs{(x_1+x_2)(x_1-(p-1)x_2)}dx_1dx_2=\tfrac{p^2+p-1}{(p+1)^2}.$$
\end{example}


\subsection{The crucial step in the recursion}
The main result of this subsection is a change of variables result, Theorem \ref{42}, which is the crucial step in performing the recursion of Theorem \ref{mainMC}. However, since the motivic recursion is inspired by a $p$-adic computation, but requires a different, more complicated, method, we begin by briefly discuss the $p$-adic method, and why it does not translate to the motivic case. This is done in the following example:
\begin{example}
Suppose we want to compute $I=\int_A\abs{\alpha_1-\alpha_2}d\underline{\alpha}$, where we integrate over the set $A\subset\bbZp^3$ of tuples such that $\alpha_1\equiv\alpha_2\mod p$ and $\alpha_1\not\equiv\alpha_3\mod p$. One way to do this is to express the integral in terms of an integral of the same function, but now integrated over all of $\bbZp^3$. This is what is done for general integrals, and in the motivic setting in Theorem \ref{42}. We illustrate this method on the integral $I$: Choosing a set of representatives of the cosets of $p\bbZp$, $x\mapsto\tilde{x}\colon\bbFp\to\bbZp$, we may write each such $\alpha_i$ uniquely as $\alpha_i=\tilde{x}_i+p\beta_i$. Moreover, $\alpha_i\equiv \alpha_j\mod p$ if and only if $x_i=x_j$ in $\bbFp$. Hence 
$$I=\sum_{\substack{\underline{x}\in\bbFp^3:\\x_1=x_2\neq x_3}}\int_{\alpha_i=\tilde{x}_i+p\beta_i}\abs{\alpha_i-\alpha_j}d\underline{\alpha}.$$
On each of these integrals we perform the change of variables $\alpha_i=\tilde{x}_i+p\beta_i$ (the $x_i$ are fixed). The Jacobian of this has absolute value $p^{-3}$, hence each of the integrals in the sum equal $p^{-3}\int_{\bbZp^3}\abs{p\beta_1-p\beta_2}d\underline{\beta}$. (This step do not generalize to arbitrary linear forms, since then the $\tilde{x}_i$ will not cancel.) Since there are $p(p-1)$ terms in the sum, we find that $I=p(p-1)p^{-3-1}\int_{\bbZp^3}\abs{\beta_1-\beta_2}d\underline{\beta}$.

The main problem in translating this to the motivic setting is that it is not possible to partition the integration set in this manner. However, this partition is not visibly in the result of the computation, so the result might still be possibly to translate into the motivic setting. This is done in Theorem \ref{42}. Also, note that in that theorem we work in the most general case possibly, contrary to in this example.
\end{example}

To prove Theorem \ref{42} we first needs two lemmas about the Witt vectors. See Appendix \ref{MC7} for notation used in connection with the Witt vectors.
\begin{lemma}\label{MC6}
Let $\f$ be a perfect field of characteristic $p$, where $p$ is a prime different from $2$, and let $\ell=aX+bY\in\W(\f)[X,Y]$ be a linear form in two variables whose coefficients are multiplicative representatives. For $x=(x_0,x_1,\dots)\in\W(A)$, where $A$ is a $\f$-algebra, define $\tilde{x}=(x_1,x_2,\dots)\in\W(A)$. Let $y\in\W(A)$ and define $\tilde{y}$ similarly. Suppose that $\ell(x,y)\equiv0\mod \V$. Then $\ell(x,y)=\V\ell(\tilde{x},\tilde{y})$. If $p=2$ the result holds provided $\ell=X-Y$.
\end{lemma}
\begin{proof}
When $p\neq2$ we have for every ring $A$ that $-1=-\R(1)=\R(-1)\in\W(A)$. This follows in the standard way by first proving it when $p$ is invertible in $A$. In this case $W^*$ is an isomorphism, and since $W^*\R(x)=(x,x^p,x^{p^2},\dots)$ we have $W^*\R(1)+W^*\R(-1)=0$ so the result follows. But if it this result is true for a ring $A$, it holds for every sub- and quotient ring, hence for every ring. 

Let $a=\R a_0$ and $b=\R b_0$. The condition $\ell(x,y)\equiv0\mod \V$ then means that $a_0x_0+b_0y_0=0$. From what was said above it follows that $a\R x_0+b\R y_0=\R(a_0x_0)+\R(b_0y_0)=\R(a_0x_0)+\R(-a_0x_0)=0$. We then have
\begin{align*}
\ell(x,y)=&\ell(\R x_0+\V\tilde{x},\R y_0+\V\tilde{y})\\
=&(a\R x_0+b\R y_0)+(a\V\tilde{x}+b\V\tilde{y})\\
=&\V(\F a\cdot\tilde{x})+\V(\F b\cdot\tilde{y})\\
=&\V\ell(\tilde{x},\tilde{y}).
\end{align*}
When $p=2$ we may still prove the result when $\ell=X-Y$, for if $\ell(x,y)\equiv0\mod\V$ then $x_0=y_0$, hence $\R(x_0)=\R(y_0)$.
\end{proof}

Throughout this section, we continue to use the notation of Subsection \ref{MC-106}: We write $\f[X_{\bullet0}\dotsk X_{\bullet N}]$ for the polynomial ring $\f[X_{i0}\dotsk X_{iN}]_{i=1}^n$, and use $Q_N$ to denote the universal polynomials defining $Q\in\DVR[X_1\dotsk X_n]$.

\begin{lemma}\label{102}
Fix a perfect field $\f$ of characteristic $p$. Let $P=\prod_{i=1}^d\ell_i\in\W(\f)[X_1\dotsk X_n]$, where the $\ell_i$ are linear forms, all of whose coefficients are multiplicative representatives, at most two of which are non-zero. Moreover, if $p=2$, assume that all the forms are of the type $X_i-X_j$. Let $x_1\dotsk x_n\in \W(A)$, where $A$ is a $\f$-algebra, be such that $\ell_i(x_1\dotsk x_n)\equiv0\mod\V$ for every $i$, and define $\tilde{x_i}$ as in Lemma \ref{MC6}. Then
$$P(x_1\dotsk x_n)=\F^{d-1}\V^dP(\tilde{x_1}\dotsk\tilde{x_n}).$$ 
In particular, working in $\W(\f[X_{\bullet n}]_{n\in\bbN})$, we have $P_\xi=0$ for $\xi<d$, and $P_{\xi+d}=P_\xi(X_{\bullet 1}\dotsk X_{\bullet\xi+1})^{p^{d-1}}$ for $\xi\in\bbN$.
\end{lemma}
\begin{proof}
Using the preceding lemma and Corollary \ref{137} we get
\begin{align*}
P(x_1\dotsk x_n)&=\prod_{i=1}^d\ell_i(x_1\dotsk x_n)\\
&=\prod_{i=1}^d\V\ell_i(\tilde{x_1}\dotsk\tilde{x_n})\\
&=\F^{d-1}\V^d\prod_{i=1}^d\ell_i(\tilde{x_1}\dotsk\tilde{x_n})\\
&=\F^{d-1}\V^dP(\tilde{x_1}\dotsk\tilde{x_n})\qedhere
\end{align*}
\end{proof}

In the case of equal characteristic, this lemma also holds, but for any set of linear forms. Its proof is straight forward, we just state the result:
\begin{lemma}\label{MC-101}
Let $\f$ be a field, and let $\DVR=\psr{\f}{t}$. Let $P=\prod_{i=1}^d\ell_i\in\DVR[X_1\dotsk X_n]$, where the $\ell_i$ are linear forms. Let $x_1\dotsk x_n\in\psr{A}{t}$, where $A$ is a $\f$-algebra, be such that $\ell_i(x_1\dotsk x_n)\equiv0\mod t$ for every $i$, and define $\tilde{x_i}$ as in Lemma \ref{MC6} (recall that we write elements of power series rings as tuples). Then
$$P(x_1\dotsk x_n)=t^dP(\tilde{x_1}\dotsk\tilde{x_n}).$$ 
In particular, working in $\psr{\f[X_{\bullet n}]_{n\in\bbN}}{t}$, we have $P_\xi=0$ for $\xi<d$, and $P_{\xi+d}=P_\xi(X_{\bullet 1}\dotsk X_{\bullet\xi+1})$ for $\xi\in\bbN$.
\end{lemma}

\begin{lemma}\label{MC-102}
Let $S$ be a finite set. Let $Q=\prod_{i\in S}\ell_i$, where the $\ell_i\in\DVR[X_1\dotsk X_n]$ are linear forms satisfying the following conditions:
\begin{itemize}
\item If $\DVR=\psr{\f}{t}$ then the $\ell_i$ are arbitrary
\item If $\DVR=\W(\f)$, where $\f$ is a field of prime characteristic different from $2$, then each $\ell_i$ is linear form in at most two variables, and its coefficients are multiplicative representatives.
\item If $\DVR=\W(\f)$, where $\f$ is a field of characteristic $2$, assume that all the forms are of the type $X_i-X_j$
\end{itemize}
Let $T$ be a subset of $S$. If $\ord\ell_i>0$ for $i\in T$, then $\ord Q\geq \abs{T}$. Moreover, for every integer $\xi\geq-1$ we have $$\mu_{\SSDVR^n}(\ord Q>\xi+\abs{T},\ord\ell_i>0,i\in T,\ord\ell_i=0,i\in S\setminus T)=\tfrac{[H_T]}{\bbL^n}\mu_{\SSDVR^n}(\ord Q>\xi),$$
where $H_T$ is the subvariety of $\SSDVR^n_1=\bbA^n_{\f}$ given by $(\ell_i)_0=0$ for $i\in T$ and $(\ell_i)_0\neq0$ for $i\in S\setminus T$.
\end{lemma}
\begin{proof}
For $N$ sufficiently large, $\pi_{N+1}(\ord Q>\xi+\abs{T},\ord\ell_i>0,i\in T,\ord\ell_i=0,i\notin T)$ is the spectrum of the algebra
\begin{equation*}
\frac{\f[X_{\bullet0}\dotsk X_{\bullet N}][(\ell_i)_0^{-1}]_{i\notin T}}{\bigl(Q_0\dotsk Q_{\xi+\abs{T}},(\ell_i)_0\bigr)_{i\in T}}.
\end{equation*}
By Lemma \ref{102} or \ref{MC-101} the class of this equals the class the spectrum of
\begin{equation*}
\frac{\f[X_{\bullet0}\dotsk X_{\bullet N}][(\ell_i)_0^{-1}]_{i\notin T}}{\bigl(Q_0(X_{\bullet1})\dotsk Q_{\xi}(X_{\bullet1}\dotsk X_{\bullet\xi+1}),(\ell_i)_0\bigr)_{i\in T}}.
\end{equation*}
Now, since the $(\ell_i)_0$ only involves the variables $X_{\bullet 0}$, we may write this as
\begin{equation*}
\frac{\f[X_{\bullet0}][(\ell_i)_0^{-1}]_{i\notin T}}{((\ell_i)_0)_{i\in T}}\otimes_{\f}\frac{\f[X_{\bullet1}\dotsk X_{\bullet N}]}{(Q_0(X_{\bullet1})\dotsk Q_{\xi}(X_{\bullet1}\dotsk X_{\bullet\xi+1}))}
\end{equation*}
The spectrum of the first factor is $H_T$ whereas, using the change of variables $X_{\bullet i}\mapsto X_{\bullet i-1}$, the spectrum of the second factor is $\pi_N(\ord Q>\xi)$. The result follows.
\end{proof}

\begin{theorem}\label{42}
Let $S$ be a finite set. For $i\in S$, let $\ell_i\in\DVR[X_1\dotsk X_n]$ be a linear forms, satisfying the conditions of Lemma \ref{MC-102}. Let $Q=\prod_{i\in S}\ell_i$. For $T\subset S$, define $Q_T=\prod_{i\in T}\ell_i$. Let $H_T$ be the subvariety of $\SSDVR^n_1=\bbA^n_{\f}$ given by $(\ell_i)_0=0$ for $i\in T$ and $(\ell_i)_0\neq0$ for $i\in S\setminus T$.  Then
$$\int_{\substack{\ord\ell_i\neq0,i\in T\\\ord\ell_i=0,i\notin T}}\abs{Q}d\mu_{\SSDVR^n}=[H_T]\bbL^{-\abs{T}-n}\int_{\SSDVR^n_\infty}\abs{Q_T}d\mu_{\SSDVR^n}.$$
\end{theorem}
\begin{proof}
By the first part of the lemma, $\{\ord Q=\xi,\ord\ell_i>0,i\in T,\ord\ell_i=0,i\notin T\}=\emptyset$ for $\xi<\abs{T}$, hence
\begin{align*}
\int_{\substack{\ord\ell_i\neq0,i\in T\\ \ord\ell_i=0,i\notin T}}\abs{Q}d\mu_{\SSDVR^n} &=\int_{\substack{\ord\ell_i\neq0,i\in T\\ \ord\ell_i=0,i\notin T}}\abs{Q_T}d\mu_{\SSDVR^n}\\
&=\sum_{\xi\geq0}\mu_{\SSDVR^n}(\ord Q_T=\xi+\abs{T},\ord\ell_i>0,i\in T,\ord\ell_i=0,i\notin T)\bbL^{-(\xi+\abs{T})}.
\end{align*}
Using the second part of the lemma it follows that this equals
\begin{equation*}
\sum_{\xi\geq0}[H_T]\bbL^{-n}\mu_{\SSDVR^n}(\ord Q_T=\xi)\bbL^{-(\xi+\abs{T})}=[H_T]\bbL^{-\abs{T}-n}\int_{\SSDVR^n_\infty}\abs{Q_T}d\mu_{\SSDVR^n}.\qedhere
\end{equation*}
\end{proof}

\subsection{The recursion}
We are finally ready to give the motivic version of the recursion. We will then see in Example \ref{MC4} how to use this in practice.
\begin{theorem}\label{mainMC}
Let $\DVR$ be a complete discrete valuation ring. Let $Q=\prod_{i\in S}\ell_i$, where $\ell_i\in\DVR[X_1\dotsk X_n]$ are linear forms, satisfying the conditions of Lemma \ref{MC-102}. Then $\int_{\SSDVR^n_\infty}\abs{Q}d\mu_{\SSDVR^n}$ is a rational function in $\bbL$, i.e., there is an $f\in\bb{Z}(T)$ such that the integral equals $f(\bbL)$. Moreover, $f$ may be computed explicitly by recursion.
\end{theorem}
\begin{proof}
This recursion is immediate, using Theorem \ref{42}: Write, for $T\subset S$, $Q_T:=\prod_{i\in T}\ell_i$. We have
$$\int_{\SSDVR^n_\infty}\abs{Q}d\mu_\SSDVR=\int_{\overline{\ell}_i=0,\,i\in S}\abs{Q}d\mu_\SSDVR+\sum_{T\subsetneq S}\int_{\substack{\overline{\ell}_i=0,\,i\in T\\ \overline{\ell}_i\neq0,\, i\in S\setminus T}}\abs{Q_T}d\mu_{\SSDVR}.$$
Theorem \ref{42} now shows that
\begin{equation}\label{MC5}
(1-[H_S]\bbL^{-\abs{S}-n})\int_{\SSDVR^n_\infty}\abs{Q}d\mu_\SSDVR=\sum_{T\subsetneq S}[H_T]\bbL^{-\abs{T}-n}\int_{\SSDVR^n_\infty}\abs{Q_T}d\mu_{\SSDVR}.
\end{equation}
The right hand side is known inductively. Moreover, since $\dim H_S\leq n$ it follows that $[H_S]\bbL^{-\abs{S}-n}\in\filtration^{\leq-\abs{S}}\KOL$, hence it is not equal to $1$. Also, by the first part of the proof of Theorem \ref{C1}, $H_S$ is an affine space, hence equal to $\bbL^m$ for some $m$. Therefore, by Proposition \ref{MC-34}, $1-[H_S]\bbL^{-\abs{S}-n}$ is invertible. Hence $\int_{\SSDVR^n_\infty}\abs{Q}d\mu_\SSDVR$ is as a rational function in $\bbL$ and the classes of various hyper plane arrangements. Because of the next theorem, Theorem \ref{C1}, it follows that the integral is a rational function in $\bbL$.
\end{proof}
The following theorem is of course already well known. We provide a proof for completeness.
\begin{theorem}\label{C1}
Let $V$ be a finite dimensional $\f$-space, $I$ a finite set and for every $i\in I$, let $\ell_i\colon V\to\f$ be a linear function. For $S\subset I$, let $H_S=\{\ell_i=0,i\in S,\ell_i\neq0,i\notin S\}\subset\bbA_V$, i.e, $H_S(R)=\set{x\in R\otimes_\f V}{\ell_i(x)=0,i\in S,\ell_i(x)\in R^\times, i\notin S}$ for every $\f$-algebra $R$. Then there is a polynomial $p\in\bbZ[X]$ such that $[H_S]=p(\bbL)\in\KOR$.
\end{theorem}
\begin{proof}
First, let $U=\cap_{i\in S}\ker\ell_i$. Then $H_S=\{\ell_i|_U\neq0,i\notin S\}\subset\bbA_U$. We may therefore assume that $S=\emptyset$. We now prove the claim by induction on the number of hyperplanes: First, let $\ell\colon V\to\f$ be non-zero. We may then choose a basis of $V$, $\{e_1\dotsk e_d\}$ such that $\ell(e_1)=1$ and $\ell(e_i)=0$ for $i>1$. Hence $\{\ell\neq0\}=\spec\f[X_1\dotsk X_d][X_1^{-1}]$, and consequently $[\{\ell\neq0\}]=\bbL^{d-1}(\bbL-1)\in\KOR$.

Assume now that the claim holds for any $\f$-space $V$ and for any collection of less than $n$ hyperplanes. Let $H=\{\ell_i\neq0\}_{i=1}^n\subset\bbA_V$. Define $U=\ker\ell_n$. Then $\{\ell_i\neq0\}_{i=1}^{n-1}\subset\bbA_V$ is the disjoint union of $H$ and $\{\ell_i|_U\neq0\}_{i=1}^{n-1}\subset\bbA_U\subset\bbA_V$, hence $[H]=[\{\ell_i\neq0\}_{i=1}^{n-1}]-[\{\ell_i|_U\neq0\}_{i=1}^{n-1}]\in\KOR$, and we are done by induction.
\end{proof}
We illustrate this with an example.
\begin{example}\label{MC4}
We apply this to the example that motivated this computations, the integral $V^n:=\int_{\SSDVR^n_\infty}\absb{\prod_{1\leq i<j\leq n}(X_i-X_j)}d\mu_{\SSDVR^n}$. By \eqref{MC5} we have
$$\bigl(1-[\{X_1=X_2\}]\bbL^{-2-1}\bigr)V^2=[\{X_1\neq X_2\}]\bbL^{-2}.$$
Since $[\{X_1=X_2\}]=\bbL$ and $[\{X_1\neq X_2\}]=\bbL^2-\bbL$ it follows that $V^2=\bbL/(\bbL+1)$.

Note that $V^2$ may also be computed using change- and separation of variables, together with the result of Example \ref{MC9}: $V^2=\int_{\SSDVR^2_\infty}\abs{X_1-X_2}d\mu_{\SSDVR^2}=\int_{\SSDVR^2_\infty}\abs{X_1}d\mu_{\SSDVR^2}=\int_{\SSDVR^1_\infty}\abs{X_1}d\mu_{\SSDVR^1} \cdot\int_{\SSDVR^1_\infty}d\mu_{\SSDVR^1}=\bbL/(\bbL+1)$.

Next we compute $V^3$. Here we really need our recursion, there is no straight forward way to compute $V^3$, as was the case with $V^2$. So we use \eqref{MC5} to obtain
\begin{multline*}
(1-[\{X_1=X_2=X_3\}]\bbL^{-3-3})V^3\\
=3[\{X_1=X_2\neq X_3\}]\bbL^{-1-3}\int_{\SSDVR^3}\abs{X_1-X_2}d\mu_{\SSDVR}+[\{X_1\neq X_2,X_2\neq X_3,X_1\neq X_3\}]\bbL^{-3}.
\end{multline*}
Here the classes of the two first hyperplane arrangements are straight forward to compute, $[\{X_1=X_2=X_3\}]=\bbL$, and $[\{X_1=X_2\neq X_3\}]=\bbL(\bbL-1)$. For the third one we use the method of Theorem \ref{C1} to obtain (the expected result) $[H_{\emptyset}]=\bbL(\bbL-1)(\bbL-2)$. Finally, by separation of variables, the integral in the right hand side equals $V^2$. Putting things together we obtain $V^3=\tfrac{(1-\bbL^{-1})(1-\bbL^{-1}+\bbL^{-2})}{(1+\bbL^{-1})(1-\bbL^{-5})}$.
\end{example}
\begin{remark}
The integrals $V^n$ comes up in connection with the problem of computing the density of the monic $n$th degree polynomials whose Galois group is the full symmetric group and whose nontrivial inertia groups are generated by a transposition. It is possible to do the same thing for general Weyl groups, and the integrals appearing in these computations are also possible to handle using the recursion.
\end{remark}

\subsection{Arbitrary forms}
In this subsection, we give a way to compute the integrals for arbitrary forms that works also in the mixed characteristic case. However, we the have to assume that the characteristic of the residue field is sufficiently large, and that the forms have coefficients in $\bbZ$. Also, this method is rather complicated to use in practice. (This method works also in the case if equal characteristic, but since we already have Theorem \ref{mainMC} in that case, we state the following theorem in the case of mixed characteristic only.)
\begin{theorem}\label{leuven1}
Let $\SSDVR=\bbA^n_{\DVR}$, where $\DVR=\W(\f)$ and $\f$ is a perfect field of characteristic $p$. Let $Q=\prod_{i\in S}\ell_i$, where $\ell_i\in\mathbb{Z}[X_1\dotsk X_n]$ are linear forms. Then, if $p$ is sufficiently large, $\int_{\SSDVR_\infty}\abs{Q}d\mu_{\SSDVR^n}\in\KOW$ is a rational function in $\bbL$, i.e., there is an $f\in\bb{Z}(T)$ such that the integral equals $f(\bbL)$. Moreover, there is an algorithm for computing $f$.
\end{theorem}
\begin{proof}
We do the recursion over the number of forms, the cardinality of $S$. Also, for the recursion to work, we compute $\int_{\overline{\mathit{m}}_j=0,\,j\in M}\abs{Q}d\mu$ for any finite set of linear functions $\mathit{m}_j$, $j\in M$ (in fact, it would suffice to compute it for the two cases when $M=\emptyset$, and when all the $\ell_j$ are contained among the $\mathit{m}_j$):

So assume that all such integrals are known when $Q$ is any product of less that $\abs{S}$ linear forms. We then want to compute $\int_{\overline{\mathit{m}}_j=0,\,j\in M}\abs{Q}d\mu$, where $Q=\prod_{i\in S}\ell_i$

As a first reduction, note that we may assume that $$K:=\bigcap_{i\in S}\ker\ell_i\cap\bigcap_{j\in M}\ker\mathit{m}_j$$ is equal to $0$. For let $K'$ be a linear complement of $K$. Let $\mathcal{Y}=\bbA_K$ and $\mathcal{Z}=\bbA_{K'}$. Then, by separation of variables, Theorem \ref{MC-15},
$$\int_{\overline{\mathit{m}}_j=0,\,j\in M}\abs{Q}d\mu_{\SSDVR}=\int_{\mathcal{Y}_\infty}d\mu_{\mathcal{Y}}\int_{\overline{\mathit{m}}_j=0,\,j\in M}\abs{Q}d\mu_{\mathcal{Z}}=\int_{\overline{\mathit{m}}_j=0,\,j\in M}\abs{Q}d\mu_{\mathcal{Z}}.$$

Hence, after a linear change of variables, Proposition \ref{MC-19}, we may assume that all the elements of the dual basis is contained among the $\ell_i$ and $\mathit{m}_j$, i.e., $\ell_i=x_i$ for $i=1\dotsk n'$ and $\mathit{m}_j=x_{n'+j}$ for $j=1\dotsk n-n'$. For this to work, we need the Jacobian to be invertible in $\bbZp$, which it is for $p$ sufficiently large.

For $T\subset S$, write $Q_T:=\prod_{i\in T}\ell_i$. Note that, for $T\subsetneq S$, the integral 
$$I_T:=\int_{\substack{\overline{\ell}_i=0,\,i\in T\\ \overline{\ell}_i\neq0,\,i\in S\setminus T}}\abs{Q_T}d\mu$$
is known by induction. For we may eliminate the $\overline{\ell}_i\neq0$ conditions in the following way: Choose $t\in S\setminus T$. Then 
$$I_T+\int_{\substack{\overline{\ell}_i=0,\,i\in T\\ \overline{\ell}_i\neq0,\,i\in S\setminus (T\setminus\{ t\})\\ \overline{\ell}_t=0}}\abs{Q_T}d\mu=\int_{\substack{\overline{\ell}_i=0,\,i\in T\\ \overline{\ell}_i\neq0,\,i\in S\setminus (T\setminus \{t\})}}\abs{Q_T}d\mu,$$
so inductively, $I_T$ may be expressed as an alternating sum of
$$\int_{\substack{\overline{\ell}_i=0,\,i\in T\\ \overline{\ell}_i= 0,\,i\in T'}}\abs{Q_T}d\mu,$$
for different $T'\subset S\setminus T$.

We now compute $\int_{\SSDVR^n_\infty}\abs{Q}d\mu$ in terms of things that are already known by induction:
$$\int_{\SSDVR^n_\infty}\abs{Q}d\mu=\int_{\overline{\ell}_i=0,\,i\in S}\abs{Q}d\mu+\sum_{T\subsetneq S}\int_{\substack{\overline{\ell}_i=0,\,i\in T\\ \overline{\ell}_i\neq0,\,i\in S\setminus T}}\abs{Q_T}d\mu.$$
Since we may assume that $\ell_i=x_i$ for $i=1\dotsk n$, we have $\{\overline{\ell}_i=0\}_{i\in S}=\{\overline{x}_i=0\}_{i=1}^n$. Hence first term of the sum is, by Theorem \ref{103}, equal to $\bbL^{-n-s}\int_{\SSDVR^n_\infty}\abs{Q}d\mu$. The rest of the terms (after the summation sign), is already known by induction. Denote this second sum with $\Sigma$. It then follows that $\int_{\SSDVR_\infty}\abs{Q}d\mu=(1-\bbL^{-s-n})^{-1}\Sigma$. (That $1-\bbL^{-s-n}$ is invertible follows from Proposition \ref{MC-34}.)

Consider now an arbitrary integral:
$$\int_{\overline{\mathit{m}}_j=0,\,j\in M}\abs{Q}d\mu=\int_{\substack{\overline{\ell}_i=0,\,i\in S\\ \overline{\mathit{m}}_j=0,\,j\in M}}\abs{Q}d\mu+\sum_{T\subsetneq S}\int_{\substack{\overline{\ell}_i=0,\,i\in T\\ \overline{\ell}_i\neq0,\,i\in S\setminus T  \\ \overline{\mathit{m}}_j=0,\,j\in M}}\abs{Q_T}d\mu$$ 
The terms after the summation sign is again taken care of by the induction assumption. The first term is, because of Theorem \ref{103}, equal to $\bbL^{-n-s}\int_{\SSDVR^n_\infty}\abs{Q}d\mu$. By the first part of the induction step, this is already known. (It would suffice to compute this integral in the case when all $\ell_i$ are contained among the $\mathit{m}_i$, and in this case we get $\int_{\overline{\mathit{m}}_j=0,\,j\in M}\abs{Q}d\mu=\bbL^{-n-s}\int_{\SSDVR^n_\infty}\abs{Q}d\mu$.)
\end{proof}

\appendix
\section{The Witt vectors}\label{MC7}
In this section we give the basic definitions in connection with the Witt vectors, $\W$. This material is essentially in \cite{MR554237} pp. 40-44 and in \cite{MR0344261}.

\subsection{Definitions}

Fix a prime $p$. For every $n\in\bbN$, define the polynomial $W_n=\sum_{i=0}^np^iX_i^{p^{n-i}}\in\bb{Z}[X]$. The ring of Witt vectors with coefficients in the commutative ring $A$, $\W(A)$, is by definition $A^\bb{N}$ with the ring operations defined by requiring that the map
\begin{align*}
W_*(A):\bff{W}(A)&\rightarrow A^\bb{N}\\
\bff{a}&\mapsto\bigl(W_0(\bff{a}),\dots,W_n(\bff{a}),\dots\bigr)
\end{align*}
should be a homomorphism. $\W(A)$ is a commutative ring with identity element $(1,0,0,\dots)$. The ring scheme of Witt vectors is the functor $\fun{\W}{\catring}{\catring}$ that takes the commutative ring $A$ to $\W(A)$. $W_*\colon\W\to\bb{A}_\bbZ^\bb{N}$ is then a morphism of ring schemes. If $p$ is invertible in $A$, $W_*(A)$ is an isomorphism, hence $\W_{\bbZ[1/p]}\simeq\bb{A}^{\bbN}_{\bbZ[1/p]}$ as ring schemes.

One defines the Witt vectors of length $n$, $\W_n$, to be the functor that takes the ring $A$ to the projection of $\W(A)$ onto its $n$ first coordinates. This scheme is of finite type over $\spec\bbZ$ . One has that $\W_1$ is the identity functor, that is $\bfW_1(A)=A$. We also have that the ring $\W(A)$ is the inverse limit of the rings $\bfW_n(A)$ as $n\rightarrow\infty$. We define the projection map $\pi_n\colon\W\to\W_n$ by
$$(a_0,a_1,\dots)\mapsto(a_0\dotsk a_{n-1})\colon\W(A)\to\W_n(A)$$
for every ring $A$.

If $A$ is a perfect ring of characteristic $p$ (meaning that $x\mapsto x^p$ is surjective) then $p$ is not a zero-divisor in $\W(A)$, which is Hausdorff and complete with respect to the filtration $\{p^n\W(A)\}_{n\in\bbN}$. Moreover, the residue ring of $\W(A)$ is $A$. In particular, $\bfW(\bbFp)=\bbZp$ and if $q=p^n$ then $\bfW(\bbFq)$ is the integral closure of $\bbZp$ in the unique unramified degree $n$ extension of $\bbQp$ (in a fixed algebraic closure of $\bbQp$).

\subsection{Operations on $\W$}

Define $\fun{\V}{\W}{\W}$ by $\V\bff{a}=(0,a_0,\dots,a_{n-1},\dots)$. $\V$ is short for "Verschiebung". It is not a morphism of ring schemes but it is additive. Note that $\W(A)/\V^n\W(A)\simeq\W_n(A)$ for every ring $A$.

Next we define the map $\fun{\R}{\bfW_1}{\bfW}$ by $a\mapsto(a,0,\dots,0,\dots)$. The map $\R$ is multiplicative. Moreover, for any $\bff{a}=(a_0,a_1,\dotsk)\in\W(A)$ we have
\begin{equation}\label{41}
\bff{a}=\sum_{i=0}^\infty\V^i\R(a_i).
\end{equation}
When $A$ is perfect of characteristic $p$, $\R(A)$ is the unique system of multiplicative representatives of $A$ in $\W(A)$. In particular, $\R(\bbFp)$ is the subset of $\bbZp$ consisting of $0$ and the $(p-1)$st roots of unity.

Finally, over $\bbFp$ (where $p$ is the prime that was fixed in the beginning of this section) we define the Frobenius morphism $\fun{\F}{\W_{\bbFp}}{\W_{\bbFp}}$ by $\F\bff{a}=(a_0^p,\dots,a_n^p,\dots)$. It is a morphism of ring schemes.
\begin{proposition}\label{prop:40}
If $A$ is an $\bbFp$-algebra and $\bff{a},\bff{b}\in\W(A)$ the following formulas hold:
\begin{align*}
\V\F\bff{a}=&\F\V\bff{a}=p\bff{a}\\
\bff{a}\cdot\V\bff{b}=&\V(\F\bff{a}\cdot\bff{b}).
\end{align*}
\end{proposition}
\begin{proof}
For the first formula see \cite{MR554237}. For the second formula it suffices to prove this when $A$ is perfect so we may assume that $\bff{b}=\F\bff{c}$. The first formula, the distributive law and the fact that $\F$ is a ring homomorphism then give
\begin{equation*}
\V(\F\bff{a}\cdot\bff{b})=\V(\F\bff{a}\cdot\F\bff{c})=\V\F(\bff{a}\cdot\bff{c})=p(\bff{a}\cdot\bff{c})=\bff{a}\cdot(p\bff{c})=\bff{a}\cdot\V\F\bff{c}=\bff{a}\cdot\V\bff{b}.\qedhere
\end{equation*}
\end{proof}
It follows that if $A$ is an $\bbFp$-algebra, $\bff{a},\bff{b}\in\W(A)$ and $i,j\in\bbN$ then
\begin{equation}\label{137}
\V^i\bff{a}\cdot\V^j\bff{b}=\V^{i+j}\bigl(\F^j\bff{a}\cdot\F^i\bff{b}\bigr).
\end{equation}
We need the following consequence of the the proposition:
\begin{corollary}\label{co:45}
Let $\f$ be a perfect $\bbFp$-algebra, let $A$ be a $\f$-algebra and let $\Delta\in\W(\f)[X_1,\dots,X_n]$ be a form of degree $d$. If $\bff{a}_1,\dots,\bff{a}_n\in\W(A)$ then
\begin{equation*}
\Delta(\V\bff{a}_1,\dots,\V\bff{a}_n)=\F^{d-1}\V^d(\F\Delta)(\bff{a}_1,\dots,\bff{a}_n).
\end{equation*}
In particular, if $\Delta\in\bbZp[X_1,\dots,X_n]$ then
\begin{equation*}
\Delta(\V\bff{a}_1,\dots,\V\bff{a}_n)=\F^{d-1}\V^d\Delta(\bff{a}_1,\dots,\bff{a}_n).
\end{equation*}
\end{corollary}
\begin{proof}
Let $\Delta=X_1^d$. The formula is true for $d=1$. Suppose that it is true for $d-1$. Then with the help of Corollary \ref{137},
\begin{align*}
\Delta(\V\bff{a})=&(\V\bff{a})(\V\bff{a})^{d-1}\\
=&(\V\bff{a})(\F^{d-2}\V^{d-1}\bff{a}^{d-1})\\
=&\V^{d}(\F^{d-1}\bff{a}\cdot\F^{d-1}\bff{a}^{d-1})\\
=&\F^{d-1}\V^d\Delta(\bff{a}).
\end{align*}

Next, let $d$ and $n$ be arbitrary and suppose the formula is proved for every $X_1^{d_1}\cdots X_{n-1}^{d_{n-1}}$ with $d_1+\dots+d_{n-1}\leq d$. Let $\Delta=X_1^{d_1}\cdots X_{n}^{d_{n}}$ with $d_1+\dots+d_n=d$. Then
\begin{align*}
\Delta(\V\bff{a}_1\dotsk\V\bff{a}_n)=&(\V\bff{a}_1)^{d_1}\prod_{i=2}^n(\V\bff{a}_i)^{d_i}\\
=&\F^{d_1-1}\V^{d_1}\bff{a}_1^{d_1}\cdot\F^{d-d_1-1}\V^{d-d_1}\prod_{i=2}^n\bff{a}_i^{d_i}.
\end{align*}
Since $\F$ and $\V$ commute we can use Corollary \ref{137} on this expression to get
$$\V^d\bigg(\F^{d-1}\bff{a}_1^{d_1}\cdot\F^{d-1}\prod_{i=2}^n\bff{a}_i^{d_i}\biggr)$$
and because $\F$ is a homomorphism this equals $\F^{d-1}\V^d\Delta(\bff{a}_1\dotsk\bff{a}_n)$.

Finally, let $\Delta\in\W(\f)[X_1\dotsk X_n]$ be an arbitrary form of degree $d$: $\Delta=\sum_{\alpha\in I}c_{\alpha}\Delta_\alpha$, where the $\Delta_\alpha$ are monomials of degree $d$ and $c_\alpha\in\W(\f)$. Since $\f$ is perfect we may, for every $\alpha\in I$, chose $c'_\alpha$ such that $\F^{d-1}c'_\alpha=c_\alpha$. We then have
\begin{align*}
\Delta(\V\bff{a}_1\dotsk\V\bff{a}_n)=&\sum_{\alpha\in I}\F^{d-1}c'_\alpha\F^{d-1}V^d\Delta_\alpha(\bff{a}_1\dotsk\bff{a}_n)\\
=&\F^{d-1}\sum_{\alpha\in I}\V^d(\F^dc'_\alpha\cdot\Delta_\alpha)\\
=&\F^{d-1}\V^d\sum_{\alpha\in I}\F c_\alpha\cdot\Delta_\alpha\\
=&\F^{d-1}\V^d(\F\Delta)(\bff{a}_1,\dots,\bff{a}_n).
\end{align*}
In particular, if $c_\alpha\in\bbZp=\W(\bbFp)$ then $\F c_\alpha=c_\alpha$, hence $\F\Delta=\Delta$.
\end{proof}

\bibliography{MC-bibliography}
\bibliographystyle{amsalpha}

\end{document}